\documentclass[12pt]{amsart}
\usepackage{graphicx,color}
\usepackage{setspace}
\usepackage{amsmath, amssymb,amsthm, amsfonts}

\usepackage[normalem]{ulem}

\def\bee{\begin{equation*}}
\def\eee{\end{equation*}}

\def\e{\epsilon}
\def\lf{\left}

\def\ijb{i\bar{j}}
\def\ri{\right}

\def\a{{\alpha}}

\def\ii{\sqrt{-1}}
\def\jbar{{\bar\jmath}}

\def\K{K\"ahler }
\def\KR{K\"ahler-Ricci }

\def\ddbar{\partial\bar\partial}
\def\be{\begin{equation}}
\def\ee{\end{equation}}

\def\lf{\left}
\def\ri{\right}
\def\a{{\alpha}}

\def\re{\text{Re}}

\def\e{\epsilon}

\def\ijb{{i\jbar}}

\def\Ric{\text{\rm Ric}}
\def\Rm{\text{\rm Rm}}

\def\re{\text{\rm Re}}

\def\re{\text{\rm Re}}

\def\ii{\sqrt{-1}}

\newtheorem{thm}{Theorem}[section]

\newtheorem{lem}{Lemma}[section]

\newtheorem{cor}{Corollary}[section]
\theoremstyle{definition}
\newtheorem{defn}{Definition}[section]
\theoremstyle{remark}

\numberwithin{equation}{section}
 \begin{document}
\title{An existence time estimate for K\"ahler-Ricci flow}
\author{Albert Chau$^1$}
\address{Department of Mathematics,
The University of British Columbia, Room 121, 1984 Mathematics
Road, Vancouver, B.C., Canada V6T 1Z2} \email{chau@math.ubc.ca}

\author{Ka-Fai Li}

\address{Department of Mathematics,
The University of British Columbia, Room 121, 1984 Mathematics
Road, Vancouver, B.C., Canada V6T 1Z2} \email{kfli@math.ubc.ca}

\author{Luen-Fai Tam$^2$}

\thanks{$^1$Research
partially supported by NSERC grant no. \#327637-06}
\thanks{$^2$Research partially supported by Hong Kong RGC General Research Fund  \#CUHK 14305114}

\address{The Institute of Mathematical Sciences and Department of
 Mathematics, The Chinese University of Hong Kong,
Shatin, Hong Kong, China.} \email{lftam@math.cuhk.edu.hk}
\thanks{\begin{it}2000 Mathematics Subject Classification\end{it}.  Primary 53C55, 58J35.}

\begin{abstract} Fix a complete noncompact \K manifold $(M^n,h_0)$ with bounded
  curvature. Let $g(t)$ be a bounded curvature solution to the \KR flow
starting from some $g_0$  which is uniformly equivalent to $h_0$. We estimate the
  existence time of $g(t)$ together with $C^0$ bounds and curvature
  bounds, where the estimates depend only on $h_0$ and the $C^0$ distance
  between $g_0$ and $h_0$. We also generalize these results to cases when
 $g_0$ may have unbounded curvature.

\end{abstract}

\maketitle\markboth{Albert Chau, Ka-Fai Li and Luen-Fai Tam} {An existence time estimate for K\"ahler-Ricci flow}
\section{Introduction}
 A motivation of this work is to understand the \KR flow on complete noncompact \K manifolds with initial metrics which may not have bounded curvature.  The \K Ricci flow is the following evolution equation for \K metric $g_0$ on an $n$ dimensional complex
manifold $M^n$

 \be\label{krf}
 \left\{
   \begin{array}{ll}
     \displaystyle\frac{\partial g_{\ijb}}{\partial t}   =-R_{\ijb} \\
     g(0)  = g_0.
   \end{array}
 \right.
 \ee
 When $g_0$ is complete with curvature bounded by $k_0$, Shi showed in \cite{Shi2} that \eqref{krf} has a short time solution $g(t)$ on $M\times[0, T)$ having bounded curvature for all $t$ (see Theorem \ref{shishortime}).  Moreover, $T$ and the curvature bound of $g(t)$ can be estimated  respectively in terms on $n, k_0$ and $n, k_0, t$.

   A fundamental problem is to determine the maximal such $T$ for a given $g_0$, and we will denote this time by $T_{g_0}$.  Extending the result of Tian-Zhang \cite{TZ} for the compact case, Lott-Zhang gave an analytic characterization of $T_{g_0}$ for general complete \K manifolds (see Theorem \ref{LottZhang}) which in particular extended Shi's result above. In this work we want to estimate $T_{g_0}$ in terms of $C^0$ norm of $g_0$.  Firstly we will use Theorem \ref{LottZhang} to show that:

 \begin{equation}\label{MT1}  C_1 h_0 \leq g_0 \leq C_2h_0 \,\,\,\,\ \implies \,\,\,\,\ C_1 T_{h_0} \leq T_{g_0} \leq C_2 T_{h_0}
\end{equation}
 where $C_i >0$ are constants and $g_0, h_0$ are complete \K metrics with bounded curvature (see Theorem \ref{t-existencetime}).    As a consequence, {\it $T_{g}$ is continuous on the space of \K metrics on $M$  in the $C^0$ norm}.  The usefulness of \eqref{MT1} comes when considering the flow of an arbitrary metric $g_0$ which is uniformly equivalent to the initial condition  $h_0$ of a known solution $h(t)$, in which case Theorem \ref{t-existencetime} tells us the existence time for $g(t)$ is uniformly comparable to that of $h(t)$.  In particular, {\it if $h(t)$ exists for all time, then so does the flow of any bounded curvature \K metric which is equivalent to $h_0$}. For example,  one can show that if $g_0$ has bounded curvature and $g_0$ is uniformly equivalent to a complete \K metric $h_0$ with bounded curvature so that the $\Ric(h_0)\le 0$, then $T_{g_0}=\infty$.

 In the above context, one may also ask about a priori estimates for the solution $g(t)$ in terms of the known solution $h(t)$.  In the context of \eqref{MT1} we will prove that: for any $0<S<C_1T_{h_0}$ we have

\begin{equation}\label{MT2}
      C_{n,S,h_0,C_1, C_2, \Lambda}^{-1} h_0\le g(t)\le C_{n,S,h_0,C_1, C_2, \Lambda}h_0 \,\,\,\,\,\, \text{on} \,\,\, M\times[0,S]
\end{equation}
where  $\Lambda=\sup_M\|Rm(g_0)\|_{g_0}$.  On the other hand, by restricting the time interval we can actually remove the dependence on the curvature of $g_0$:  for any $0<s<S<\frac{C_1}{n}T_{h_0}$ we have
\begin{equation}\label{MT3}
      C_{n,s,S,h_0,C_1, C_2}^{-1} h_0\le g(t)\le C_{n,s,S,h_0,C_1, C_2}h_0 \,\,\,\,\,\, \text{on} \,\,\, M\times[s,S]
\end{equation}
See Theorem \ref{t-C0-bound} for these estimates.   As mentioned above in \cite{Shi2} Shi proved the estimates in  \eqref{MT2}, but with the upper time limit $C_1T_{h_0}$ replaced by some $T(n, \Lambda)$ depending only on $n, \Lambda$ (in this case, as in \cite{Shi2}, one may just take $h_0=g_0$).   The point is that in many cases, $C_1T_{h_0}$ will be strictly larger than  $T(n, \Lambda)$, thus leading to a better estimate of $T_{g_0}$ and corresponding estimates.

 From \eqref{MT1} and \eqref{MT2} we may then conclude higher order estimates as in Theorem \ref{t-higherorder}.   One can then apply these estimates to study \KR flow with initial metric which may have unbounded curvature and existence of solutions $g(t)$ starting from $C^0$ limits of bounded curvature metrics uniformly equivalent to some fixed $h_0$  with bounded curvature (Theorem \ref{t-generalization}).   This  generalizes some of our earlier results in \cite{ChauLiTam}.

 The paper is organized as follows.  In \S2.1 we prove the existence time comparison in \eqref{MT1}.  In \S2.2 we establish the comparison in \eqref{MT2}, then establish the corresponding higher order estimates \eqref{MT2}, \eqref{MT3} in \S2.3.  In \S3 we generalize to initial metrics $g_0$ which may only be Hermitian continuous.  Finally we collect some known results, which we will use throughout, in the Appendix.

\section{existence time and estimates}

Suppose $(M^n,g_0)$ is a complete noncompact \K manifold with bounded curvature, and let $g(t)$ be a solution to \KR flow \eqref{krf} on $M\times[0,T)$ with $g(0)=g_0$.  We make the following:

\begin{defn}
\begin{enumerate}
We call the solution $g(t)$
\item {\it a bounded curvature solution starting from $g_0$} if $g(t)$ has uniformly bounded curvature on $M\times[0,S]$ for all $S<T$;
\item {\it the maximal bounded curvature solution starting from $g_0$} if $g(t)$ satisfies the condition in (1) and $T$ is maximal with this property, in which case we denote $T=T_{g_0}$.
\end{enumerate}
\end{defn}

By Theorem \ref{uniqueness}, a bounded curvature solution starting from $g_0$, on $M\times[0, T)$, is indeed the unique solution with bounded curvature.

 Now let $(M^n, h_0)$ be a fixed complete noncompact \K manifold with bounded curvature.  We define the following spaces of \K metrics.
\begin{defn} For any $0<a<b$ and $\Lambda>0$,   we define the sets
$$  \mathcal{G}(h_0;a, b, \Lambda)\subset \mathcal{G}(h_0;a, b) \subset  \mathcal{G}(h_0;a) $$
as follow

 \begin{enumerate}

\item $\mathcal{G}(h_0;a)$ is the set of bounded curvature \K metrics $g_0$ on $M$ such that $ah_0\le g_0\le C(g_0)h_0$ for some constant $C(g_0)$ which may depend on $g_0$
\item $ \mathcal{G}(h_0;a, b)$ is the set of bounded curvature \K metrics $g_0$ on $M$ such that $ah_0\le g_0\le b h_0$
\item $ \mathcal{G}(h_0;a, b, \Lambda)$ is the set of \K metrics $g_0$ on $M$ with sectional curvatures bounded by
$ \Lambda$ and $ah_0\le g_0\le b h_0$
\end{enumerate}
\end{defn}

  For $g_0$ in each of the above classes, we want to estimate $T_{g_0}$ and various bounds on the corresponding maximal bounded curvature solution.  First we need the following result by Lott-Zhang \cite{Lott-Zhang}.

\begin{thm}(Lott-Zhang \cite{Lott-Zhang}) \label{LottZhang} Let $g_0$ be a complete \K metric with bounded curvature on a non-compact manifold $M$.  Then $T_{g_0}$ is equal to the supremum of the
numbers $T$ for which there is a bounded function $F_{T}\in C^{\infty}(M)$
such that
\begin{enumerate}
\item [(i)] $g_0-T\Ric(g_0)+\sqrt{-1}\partial\bar{\partial}F_{T} \geq c_T g_0$ for some $c_T >0$
 \item[(ii)]  $|F_T|$ and the quantities $|\nabla^l\Rm(g_0)|_{g_0}$, $|\nabla^l\ddbar F_T|_{g_0}$, for $0\le l\le 2$, are uniformly bounded on $M$.
\end{enumerate}
  Moreover if $T$ satisfies (i) and (ii), then for any $T'<T$, there is a constant $C$ depending $T'$, $c_T$, and the bound on the quantities in (ii) such that
$$
C^{-1}g_0\le g(t)\le Cg_0
$$
on $M\times[0,T']$.
\end{thm}
\begin{proof}

Suppose the above conditions are satisfied for some $T<\infty$.  If $0<T<T_{g_0}$ then the proof Theorem 4.1 in \cite{Lott-Zhang}  shows the existence of a bounded $F_T$ which satisfies (i) and also has bounded covariant derivatives of all orders, thus also satisfying (ii).  Conversely, given $F_T$ satisfying (i) and (ii),   by the estimates in the proof of \cite[Theorem 4.1]{Lott-Zhang} (up to and including Lemma 4.26)  there exists constants  $C_1, C_2>0$ such that $C_1g\leq g(t)\leq C_2 g$ for all $t\in [0, \min(T_{g_0}, T))$. Then by Theorem \ref{t-EvanKrylov}, we have that the curvature of $g(t)$ is uniformly bounded on $M\times[0, \min(T_{g_0}, T))$ which implies $0<T<T_{g_0}$ by Theorem \ref{shishortime}.    This completes the proof that $T_g$ is characterized by (i) and (ii).

Moreover, from the above mentioned estimates in \cite{Lott-Zhang}, one can see that the last statement in the Theorem also holds.
\end{proof}

\subsection{existence time estimate}

The first result of this section is the following existence time estimate:

\begin{thm}\label{t-existencetime}
  $T_{g_0}\ge aT_{h_0}$ for all $g_0\in \mathcal{G}(h_0;a)$.
  \end{thm}
  \begin{proof} First observe that if $\lambda>0$ is a constant, {\re then} $T_{\lambda h_0}=\lambda T_{h_0}$. Hence without loss of generality, we may assume that $a=1$.  Also, we assume without loss of generality that  $g_0 , h_0$ have bounded geometry of order $\infty$.  For if not, we let $g(t), h(t)$ are the corresponding solutions as in Theorem \ref{shishortime}, then we first prove the Theorem for $g(\e), h(\e)$ for arbitrary small $\e$ and then let $\e\to 0$ and use the fact that $g(\e), h(\e)$ converge uniformly to $g_0, h_0$ (respectively) on $M$ by Corollary  \ref{c-shishortime}, and $T_{g(\e)}=T_g-\e$ and $T_{h(\e)}=T_h-\e$ by Theorem \ref{uniqueness}. Note that by Theorem \ref{t-EvanKrylov}, all the covariant derivatives of $g(\e)$ with respect to $h_0$ are bounded. So we may assume   in addition  that all the covariant derivatives of $g_0$ with respect to $h_0$ are bounded.

  Now for any  $0<T< T_{h_0}$, we have for $g_0\in \mathcal{G}(h_0;a)$,
\bee
\begin{split}
g_0-T\Ric(g_0)&= h_0-T\Ric(h_0)+T(\Ric(g_0)-\Ric(h_0))+ (g_0-h_0).
\end{split}
\eee
By  Theorem \ref{LottZhang}, there is a smooth bounded function $f$ with bounded covariant derivatives with respect to $h_0$ such that
$$
h_0-T\Ric(h_0)+\ii\ddbar f\ge C_1h_0
$$
for some $C_1>0$.  Then letting $F=\log \frac{\omega_0^n}{\eta^n}$ where $\omega_0$ and $\eta$ are the \K forms of $g_0$ and $h_0$ respectively, gives

\bee
\begin{split}
g_0-T\Ric(g_0)+\ii\ddbar(f+TF) \ge &C_1h_0 + (g_0-h_0)\\
\ge &C_2g_0
\end{split}
\eee
 for some $C_2>0$ because $g_0\ge h_0$ and $g_0$ is uniformly equivalent to $h_0$.    From the facts that $g_0, h_0$ are equivalent and that all the covariant derivatives of $g_0$ with respect to $h_0$ are bounded, we may conclude that all the covariant derivatives of $f$ are bounded with respect to $g_0$ as well.  We may also conclude from these facts that $F$ and $|\nabla^l\ddbar F|_{g_0}$ are uniformly bounded for $0\le l\le 2$ where we have used that $\ii\ddbar F=-\Ric(g_0)+\Ric(h_0)$.  By Theorem \ref{LottZhang}, we conclude that $T\le T_{g_0}$. From this the result follows.
 \end{proof}

By the theorem, we have   the following monotonicity and continuity of $T_{g}$. Namely

\begin{cor}\label{c-existencetime}
\begin{enumerate} Let $M^n$ be a noncompact complex manifold.
  \item [(i)] Let   $g_0\ge h_0$ be complete uniformly equivalent \K metrics on $M$ with bounded curvature. Then  $T_{g_0}\ge T_{h_0}$.  In particular, if $T_{h_0}=\infty$, then $T_{g_0}=\infty$.

  \item [(ii)] Let $\mathcal{K}$ be the set of complete \K metrics on $M$ with bounded curvature. Then $T_{g}$  is continuous on $ \mathcal{K}$ with respect to the $C^0$ norm in the following sense: Let $g_0\in \mathcal{K}$. Then for $h_0\in  \mathcal{K}$, $T_{h_0}\to T_{g_0}$  as $||h_0-g_0||_{g_0}\to 0$.
  \end{enumerate}
      \end{cor}
 We can extend the existence time  estimate  in Corollary
4.2 \cite{ChauLiTam}.   Namely,  we can estimate the existence time in terms of upper bound of Ricci curvature instead of upper bound of the bisectional curvature.
\begin{cor}\label{RicciK}
Let $(M,h)$ be a complete K\"ahler metric having bounded curvature with Ricci curvature bounded
above by $K$, and let $g\in\mathcal{G}(h;1)$.  then $T_g\geq\frac{1}{K}$.\end{cor}
\begin{proof}
Since $\Ric(h)\leq K h$, so $h-t\Ric(h)=(1-Kt) h$.
By Theorem \ref{LottZhang}, $T_h\geq\frac{1}{K}$. By Theorem
\ref{t-existencetime}, $T_g\geq\frac{1}{K}.$ \end{proof}

\subsection{$C^0$ estimates}

 In this section we will establish $C^0$ bounds for maximal solutions $g(t)$ starting from some $g_0 \in \mathcal{G}(h_0;a,b)$, where the bounds are uniform over all such $g_0$.  We will establish similar bounds for maximal solutions $g(t)$ starting from some $g_0 \in \mathcal{G}(h_0;a,b, \Lambda)$.  We begin with the following:

 \begin{lem}\label{l-dtensor} Let $g_0$ and $h_0$ be complete noncompact \K metrics on $M^n$ with curvature bounded by $k_0$ such that $A^{-1}g_0\le h_0\le Ag_0$. Let $T(n,k_0)>0$ be the existence time of \KR flows with initial data $g_0, h_0$ in Theorem \ref{shishortime}. For any $0<T<T(n,k_0)$, and for any tensor $\Theta$ on $M$. Suppose all the covariant derivatives of $\Theta$ with respect to $h(T)$ are bound. Then the covariant derivatives of $\Theta$ with respect to $g(T)$ are also bounded. Namely, for any $l\ge 0$
$$
 |\nabla_{g(T)}^l\Theta|_{g(T)}\le D_l
$$
where $D_l$ is a constant depending only on $k_0, n, T, A, l$ and the bounds of $|\nabla_{h(T)}^k\Theta|_{h(T)}$ with $0\le k\le l$.

\end{lem}
\begin{proof} By Theorem \ref{shishortime}, the curvature and all the covariant derivatives up to order $l$ of both $h(t), g(t)$ are uniformly bounded on $M\times [\frac T2,T]$ by a constant $C_1$ depending only on $  n, k_0, T, l$. Hence by the \KR flow equations, we have $C_2^{-1}h(T)\le g(t)\le C_2h(T)$ for $t\in [\frac T2,T]$, where $C_2$ is a constant depending only on $n,k_0,T$. By Theorem \ref{t-EvanKrylov}, we conclude that
$$
|\nabla_{h(T)}^kg(T)|_{h(T)}\le C_3
$$
for all $0\le k\le l$ for some constant $C_3$ only on $n,k_0,T,l$. From this, it is easy to see the lemma is true.
\end{proof}

\begin{thm}\label{t-C0-bound} Let $h_0$ be as before.
\begin{enumerate}
  \item [(1)] For any $0<a<b$ and  $\Lambda>0$, and any $0<S<aT_{h_0}$ there is a constant $C$ depending only on $n,S,h_0,a,b, \Lambda$ such that
      $$
      C^{-1}h_0\le g(t)\le Ch_0
      $$
      on $M\times[0,S]$ for   all maximal solutions $g(t)$ starting from some $g_0\in \mathcal{G}(h_0;a,b,\Lambda).$

  \item [(2)]  For any $0<a<b$, and any $0<S<\frac anT_{h_0}$ there is a constant $C$ depending only on $n, s, S,h_0,a,b $ such that
      $$
      C^{-1}h_0\le g(t)\le Ch_0
      $$
      on $M\times[0,S]$ for   all maximal solutions $g(t)$ starting from some $g_0\in \mathcal{G}(h_0;a,b).$
\end{enumerate}

\end{thm}

\begin{proof}As before, we may assume that $a=1$.  We first prove part (1).    Let $h(t)$ be the maximal solution starting from $h_0$ and let $g(t)$ be the maximal solution $g(t)$ starting from some $g_0\in \mathcal{G}(h_0;a,b,\Lambda)$.   Let   $0<S<T_{h_0}$  be given.   Fix $\a<1$ such that $\a T_{h_0}>S$. Then $T_{\a h_0}=\a T_{h_0}>S$. We claim that there is $\e>0$ depending only on $n,S,h_0,a,b, \Lambda$ with $S<\a T_{h_0}-\e$  and satisfies the following:

 \begin{enumerate}
  \item [(i)]  $ C^{-1}h_0\le g(t)\le Ch_0
      $ for all $t\leq \e$ and some $C>0$ depending only on $n, \Lambda$;
\item [(ii)]    
 $ 
 \hat h(t)=\a h(\a^{-1}(t+\e))
 $ is the solution  the \KR flow equation   on $M\times[0,S]$ with initial data $\hat h(0)=\a h(\a^{-1}\e)$ such that
 $$
 2b\a^{-1}\hat h(0)=2bh(\a^{-1}\e )\ge g(\e)>\a h(\a^{-1}\e )=\hat h(0);$$ 
 \item[(iii)] and
  there exists a constant $C_1>0$ and a smooth function $f$ with
 \bee
    \hat h(0)-S\Ric(\hat h(0))+\ii\ddbar f\ge C_1 \hat h(0)\eee
    where $C_1$ and the covariant derivatives of $f$ with respect to  $h(\frac\e\a)$ are bounded by constants depending only on $n, h_0, \e, \a$.
\end{enumerate}
 
Indeed, by Theorem \ref{shishortime} and Corollary \ref{c-shishortime} we may choose $\e$ depending only on $n$ and  $\Lambda$ so that  (i) and (ii)  hold because
$S<\a T_{h_0}-\e=T_{\hat h(0)}$.       (iii)  holds for some $C_1, f$ by Theorem \ref{LottZhang}  applied to the solution $\hat h(t)$.

Then letting $F=\displaystyle{\log \frac{\det( \a h(\a^{-1}\e))}{\det(g(\e))}}$ and using (ii) and (iii) 
\begin{equation}\label{S-e}
g(\e)-S\Ric(g(\e))+\ii\ddbar(f+SF) \geq C_2 g(\e)
\end{equation}
for some $C_2>0$ depending only on $\e, a, b, \Lambda$.  Now $\|\nabla^l\Rm(g(\e))\|_{g(\e)}$ and $\|\nabla^l\Rm(h(\e))|_{h(\e)}$ are bounded by constants depending only on $n, h_0, S, a, b, \Lambda$.   Using (i) and Lemma \ref{l-dtensor},  we may conclude that the covariant derivatives of $(f+SF)$ with respect to $g(\e)$ are all bounded by constants depending only on $n, h_0, \Lambda, S, a, b$.   From this and \eqref{S-e}, the last statement in Theorem \ref{LottZhang} gives $ C^{-1}h_0\le \hat g(t+\e) \le Ch_0$ for all $t\leq S$ and some $C$ depending only on $n, h_0, S, a, b, \Lambda$.  We conclude from this and (i) that (1) is true.  This completes the proof of (1).

 To prove (2), by \cite{ChauLiTam} if $g(t)$ is the \KR flow with bounded curvature so that $g(0)=g_0\in \mathcal{G}(h_0;1,b)$, then
 $$
 (\frac1n-2Kt)h_0\le g(t)\le B(t) h_0
 $$
 on $M\times[0,\frac{1}{2nK})$ for some continuous function $B(t)>0$ on $[0, \frac1{2nK})$ which depends only on $n, K$. Here $K\ge 0$ is an upper bound for the bisectional curvature of $h_0$. By Theorem \ref{t-EvanKrylov} and   \ref{shishortime}, one can see that the (2) follows from (1), by considering $g(\e)$ for $\e>0$ small.
\end{proof}

\subsection{higher order estimates}

By the local Evans-Krylov Theory in \cite{E,K,SW}, more precisely by Theorem \ref{t-EvanKrylov}, we may have the following estimates based on the $C^0$ estimates in Theorem \ref{t-C0-bound}.

\begin{thm}\label{t-higherorder} Let $h_0$ be as before.
\begin{enumerate}
  \item [(1)] For any $0<S<aT_{h_0}$ and for all $l\ge0$, there is a constant $C_l$ depending only on $n,S,h_0,a,b, \Lambda$ and $l$.
      \bee \displaystyle \sup_{M\times[0, S]}\| \nabla_{g(t)}^l \Rm(g(t))\|^2_{g(t)} \leq \frac{C_l}{t^l}
       \eee
        for all maximal solutions $g(t)$ starting from some $g_0\in \mathcal{G}(h_0;a,b,\Lambda).$

  \item [(2)]  For any $0<s<S<\frac anT_{h_0}$ there is a constant $C$ depending only on $n, s, S,h_0,a,b $ such that
      \bee \displaystyle \sup_{M\times[s, S]}\| \nabla_{g(t)}^l \Rm(g(t))\|^2_{g(t)} \leq \frac{C}{t^{l+2}}
       \eee
         for all maximal solutions $g(t)$ starting from some $g_0\in \mathcal{G}(h_0;a,b).$
\end{enumerate}

\end{thm}

\begin{proof} (1) By Theorem \ref{shishortime}, there is a constant $T(n,\Lambda)>0$ such that the conclusion is true for $S<T(n,\Lambda)$. For $aT_{h_0}>S\ge T(n,\Lambda)$,
this follows directly from the $C^0$ bounds in Theorem \ref{t-C0-bound}(1) and  Theorem \ref{t-EvanKrylov}.  The bounds   also follow from the estimates in  \cite{Yu} and the $C^0$ estimates.

(2) also follows from Theorem \ref{t-C0-bound}(2) and  Theorem \ref{t-EvanKrylov}.
\end{proof}

\section{generalizations}

In this section, we will generalize some of the results Theorems \ref{t-existencetime}, \ref{t-C0-bound}, \ref{t-higherorder} without assuming that $g_0$ has bounded curvature. Let us introduce more spaces. Let $(M^n,h_0)$ be a complete noncompact \K manifold with bounded curvature. Let $0<a<b$ be constants. Let  $\text{Cl} \lf(\mathcal{G}(h_0;a,b)\ri)$ be the set all  continuous Hermitian metrics $g_0$ on $M$ such that there exist a sequence $g_{i}\in \mathcal{G}(h_0;a,b)$ with $g_i\to g_0$ uniformly on $M$. Namely, $\sup_M|g_i-g_0|_{h_0}\to 0$ as $i\to\infty$. Let $\text{Cl}_{\text {loc}}\lf(\mathcal{G}(h_0;a,b)\ri)$ be the set all  continuous  Hermitian metrics $g_0$ on $M$ such that there exist a sequence $g_{i}\in \mathcal{G}(h_0;a,b)$ with $g_i\to g_0$ uniformly on compact sets in $M$. For $\Lambda>0$, one can define $\text{Cl}_{\text loc}\lf(\mathcal{G}(h_0;a,b,\Lambda)\ri)$ similarly.

\begin{thm}\label{t-generalization}
Let $(M, h_{0})$ be a complete noncompact \K manifold with bounded curvature and let $0<a<b$ be constants.
\begin{enumerate}
\item
Given any $g_0 \in \text{Cl}_{\text {loc}}\lf(\mathcal{G}(h_0;a,b)\ri)$, there is a solution $g(t)$ to the \KR flow \eqref{krf} on $M\times(0,\frac an T_{h_0})$ such that
\begin{enumerate}
\item[(i)] $g(t)$ is \K for $t\in   (0,\frac anT_{h_0})$ and $g(t) \to g_0$ as  $t\to0$ uniformly on compact sets of $M$;
\item[(ii)] for any $0<S<\frac an T_{h_0}$, there is a constant $C>0$ such that $$C^{-1}h_0\le g(t)\le Ch_0$$ on $M\times[0,S]$
\item[(iii)] for any $0<s<S<\frac anT_{h_0}$, and for any $l\ge0$, we have $$|\nabla^l_{g(t)}\Rm(g(t))|_{g(t)}^2\le \frac{C}{t^{l+2}}.$$
\end{enumerate}

\item If $g_0 \in \text{Cl}_{\text loc}\lf(\mathcal{G}(h_0;a,b,\Lambda)\ri)$ for some $\Lambda>0$, then the conclusion in (1) holds with $\frac an T_{h_0}$ replaced by $aT_{h_0}$.

\item If $g_0\in \text{Cl} \lf(\mathcal{G}(h_0;a,b)\ri)$, then the conclusion in (1) holds with $\frac an T_{h_0}$ replaced by $aT_{h_0}$.

\end{enumerate}

 In particular, if $T_{h_0}=\infty$, then $g(t)$ above exists for all time $t>0$ such that the estimates above hold.
\end{thm}

\begin{proof} We first prove (1).  By Theorem 4.2 in \cite{ChauLiTam}, there is some $T>0$ such that the conclusions in (1) hold when replacing $\frac anT_{h_0}$ with $T$.  Assume that $T$ is maximal such that this is true.  By the proof of Lemma 3.1 in \cite{ChauLiTam}, there are continuous functions $A(t), B(t)$ on $[0, T)$
such that $A(t)h_0\leq g(t)\leq B(t)h_0$ with $A(t)\to \frac{a}{n}$
as $t\to0$.  Applying Theorem \ref{t-higherorder} (1) to $g(\e)$ for small $\e$ implies that $T\geq  \e+A(\e)T_{h_0}$.  We conclude that $T\geq \frac anT_{h_0}$ by letting $\e \to 0$.

We now prove part (2).  Let $ \{g_{i,0}\}_{i=1}^{\infty}\subset \mathcal{G}(h_0;a, b,\Lambda)$ such that $g_{i,0}\to g_0$ uniformly in compact sets of $M$.  Then by Theorem \ref{t-C0-bound} (2) and Theorem \ref{t-higherorder} (2), we have a sequence of solutions $g_i(t)$ to \eqref{krf} with initial data $g_{i,0}$ on $M\times(0, aT_{h_0})$ converging to a smooth limit solution $g(t)$  to \eqref{krf} on $M\times(0,a T_{h_0})$.  Moreover, for any $0<S<  a T_{h_0}$, there is a constant $C>0$ independent of $i$ such that $C^{-1}h_0\le g_i(t)\le Ch_0$ on $M\times(0,S]$, and thus there exists a smooth limit solution $g_i(t)\to g(t)$ on $M\times(0,S]$ satisfying (ii).  From this we may prove that $g(t)$ satisfies (i) as in the proof of Theorem 4.1 in \cite{ChauLiTam}.  Moreover by applying Theorem \ref{t-EvanKrylov}  we may also conclude that $g(t)$ satisfies (iii).

To prove (3), let  $g_0\in \text{Cl} \lf(\mathcal{G}(h_0;a,b)\ri)$. By (1), we obtain $g(t)$ satisfying the conclusions in (1). We need to prove that $g(t)$ can be extended up to $aT_{h_0}$ so that the estimates in (1) are true up to $aT_{h_0}$. By the proof of \cite[Theorem 4.1]{ChauLiTam}, we have $g(t)\to g_0$ uniformly on $M$.  Since $ah_0\le g_0\le bh_0$ and $g(t)$ has bounded curvature for $t>0$, by  applying Theorem \ref{t-existencetime}, \ref{t-C0-bound}, \ref{t-higherorder}, the result follows.

\end{proof}

\begin{cor}\label{c-integralbound} Let $(M^n,h_0)$ be a complete noncompact \K manifold with bounded curvature. Suppose $g_0$ is another \K metric on $M$ such that (i) $ah_0\le g_0\le bh_0$ for some $0<a<b$; (ii) $\Ric(g_0)\ge -Kg_0$ for some $K$; and (iii) there exist $r_0>0$, $p_0>n$ and a constant $C$ with
$$
\frac1{V_{g_0}(x,r_0)} \int_{B_{g_0}(x,r_0)}|\Rm(g_0)|^{p_0}\le C
$$
for all $x\in M$. Then there is a solution $g(t)$ to the \KR flow with $g(0)=g_0$ on $M\times[0,aT_{h_0})$ such that for any $0<S<aT_{h_0}$ there is a constant $C$ with
$C^{-1}h_0\le g(t)\le Ch_0$ on $M\times[0,S]$.

\end{cor}
\begin{proof} By the results in \cite{Xu,HuangTam}, there is a short time solution $g(t)$ to the \KR flow \eqref{krf} with $g(0)= g_0$. Moreover, the curvature has the following bound:
$$
|\Rm(g(t))|\le C_1t^{-\frac n{p_0}}
$$
for some constant $C_1$. Since $p_0>n$, we have $g(t)\to g_0$ as $t\to0$ uniformly on $M$. One can then apply Theorem \ref{t-generalization} to conclude that the corollary is true.

\end{proof}

\appendix

\section{}

In this appendix we collect some known results for easy reference. The first result the basic existence theorem of W.-X. Shi \cite{Shi2}:

\begin{thm}\label{shishortime}  Suppose $g_0$ is a complete \K metric on
a noncompact complex manifold $M^n$ with complex dimension $n$ with curvature bounded by a constant $k_0$.  Then
there exists $0<T\leq \infty$ depending only on $k_0$ and the dimension $n$, and
a smooth solution $g(t)$ to \eqref{krf} on $M \times [0, T)$ with initial
condition $g(0)=g_0$ such that
 \begin{enumerate}

 \item[(i)] $g(t)$ has uniformly bounded curvature on $M\times[0,T']$ for all
$0<T'<T$.  More generally, for any $l\geq 0$ there exists a constant $C_l$
depending only on $l, k_0, n, T'$   such that
     $$
 sup_M |\nabla^l \Rm(g(t))|^2_{g(t)}\le \frac {C_l}{t^{l}},
 $$
 on $M\times[0,T']$.
     \item[(ii)] If $T<\infty$ and $\displaystyle \lim_{t\to T} \sup_{M} |Rm(x,
t)|<\infty$, then the $g(t)$ can be extended, as a solution to \eqref{krf}, beyond $T$ to $T_1>T$ so that
(i) is still true with $T$ replaced by $T_1$.
         \end{enumerate}
\end{thm}

The following is an immediate consequence of the theorem:
\begin{cor}\label{c-shishortime}   Let $g(t)$ be as in Theorem \ref{shishortime}.  Then
\begin{enumerate}

\item[(i)] $g(t)$ is \K and equivalent to $g_0$ for all $t\in [0, T)$, namely for all $T'<T$, there a constant $C>0$ depending only on  $k_0, n, T'$ such that
     $$
     C^{-1}g_0\le g(t)\le C g_0
     $$
     on $M\times[0,T']$
\item [(ii)]   There is a continuous function $\e(t)>0$ depending only on $n, k_0$ with $\e(t)\to0$ as $t\to 0$  such that
$$
(1+\e(t))^{-1}g_0\le g(t)\le (1+\e(t))g_0.
$$

\end{enumerate}
\end{cor}
Using Evan-Krylov theory \cite{E,K} or using maximum principle for \KR flow by Sherman-Weinkove \cite{SW}, we have the following (see \cite{ChauLiTam}):
 \begin{thm}\label{t-EvanKrylov}  Let $(M^n,h_0)$ be a complete noncompact
\K manifold with bounded geometry of infinite order. Let $g(t)$ be a solution
of the \KR flow \eqref{krf} on $M\times[0,T)$ with initial data $g_0$ which is a
complete \K metric.  For any $x\in M$, suppose there is a constant $N>0$, such
that
 $$
 N^{-1}h_0\le g(t)\le Nh_0
 $$
 on $ B_{h_0}(x,1)\times [0,T)$.  Then

  \begin{itemize}
    \item [(i)]
 $$
 | \nabla_{h_0}^k g(t)|^2_{h_0}\leq \frac{C_k}{t^{k}}
 $$
 on $  B_{h_0}(x,1/2)\times(0, T)$, for some constant $C_k$ depending only on $k$,
 $n$, $T$ and $N$, and the bounds of $|\nabla^l\Rm(h_0)|_{h_0}$.

 \item[(ii)] If we assume   $|\nabla_{h_0}^k g_0|_{h_0}^2$  is bounded in $
B_{h_0}(x,1)$ by $c_k$, for $k\ge 1$,  then
  $$
  |\nabla_{h_0} ^k g(t)|^2_{h_0}\le {C_k}  ,
 $$
  on $  B_{h_0}(x,1/2)\times[0, T)$ for some constant $C_k$ depending only on $k$,
$c_k$, $n$, $T$ and $N$.
  \end{itemize}
\end{thm}

We also need the following uniqueness theorem on \KR flow \cite{ChauLiTam2}:

\begin{thm}\label{uniqueness}
Let $(M^n,h_0)$ be a complete noncompact \K manifold with bounded curvature.
 Let $g_1(x,t)$ and $g_2(x,t)$ be two solutions of the \KR flow \eqref{krf} on $M\times[0,T]$ with the same initial data $g_0(x)=g_1(x,0)=g_2(x,0)$. Suppose there is a constant $C$ such that:
 \bee
 C^{-1}h_0\le  g_1(t),    g_2(t)\le Ch_0
\eee
on $M\times[0,T]$. Then $g_1\equiv g_2$ on $M\times[0,T]$.
 \end{thm}
.

\end{document}